\documentclass[12pt,a4paper]{amsart}

\usepackage[utf8]{inputenc}
\usepackage{enumerate,a4,titletoc}
\usepackage{bbm,mathtools,amssymb,xspace}
\usepackage[usenames,dvipsnames]{xcolor}
\usepackage[pagebackref,colorlinks,linkcolor=BrickRed,citecolor=OliveGreen,urlcolor=black,hypertexnames=true]{hyperref}
\usepackage[margin=2.5cm]{geometry}

\DeclareMathOperator{\tr}{tr}

\DeclareMathOperator{\End}{End}
\DeclareMathOperator{\ran}{ran}

\DeclareMathOperator{\GL}{GL}

\DeclareMathOperator{\dom}{dom}
\DeclareMathOperator{\real}{Re}
\DeclareMathOperator{\imag}{Im}
\DeclareMathOperator{\opm}{M}
\DeclareMathOperator{\oph}{H}
\newcommand{\ve}{\varepsilon}
\newcommand{\de}{\delta}
\newcommand{\N}{\mathbb{N}}
\newcommand{\R}{\mathbb{R}}
\newcommand{\C}{\mathbb{C}}

\newcommand{\cK}{\mathcal{K}}

\newcommand{\cS}{\mathcal{S}}

\newcommand{\cV}{\mathcal{V}}

\newcommand{\rr}{\mathbbm r}

\newcommand{\ulx}{\boldsymbol{x}}

\newcommand*{\mat}[1]{\opm_{#1}(\C)}
\newcommand*{\gl}[1]{\GL_{#1}(\C)}
\newcommand*{\herm}[1]{\oph_{#1}(\C)}
\newcommand{\all}{\mathbb{M}^d}
\newcommand{\allh}{\mathbb{H}^d}
\newcommand{\alld}{\mathbb{D}^d}

\newcommand{\Langle}{\mathop{<}\!}
\newcommand{\Rangle}{\!\mathop{>}}

\newcommand{\px}{\C\!\Langle \ulx\Rangle}
\newcommand{\pxx}{\C\!\Langle \ulx,\ulx^*\Rangle}

\makeatletter
\def\moverlay{\mathpalette\mov@rlay}
\def\mov@rlay#1#2{\leavevmode\vtop{
		\baselineskip\z@skip \lineskiplimit-\maxdimen
		\ialign{\hfil$#1##$\hfil\cr#2\crcr}}}
\makeatother

\newcommand{\plangle}{\moverlay{(\cr<}}
\newcommand{\prangle}{\moverlay{)\cr>}}
\newcommand{\rx}{\C\plangle \ulx \prangle}

\newcommand{\sstable}{purely stable\xspace}
\newcommand{\strstable}{S-stable\xspace}
\newcommand{\irr}{indecomposable\xspace}
\newcommand{\bqed}{\hfill $\blacksquare$}

\newtheorem{thm}{Theorem}[section]
\newtheorem{lem}[thm]{Lemma}
\newtheorem{cor}[thm]{Corollary}
\newtheorem{prop}[thm]{Proposition}
\newtheorem{thmA}{Theorem}

\theoremstyle{definition}
\newtheorem{defn}[thm]{Definition}
\newtheorem{exa}[thm]{Example}

\theoremstyle{remark}
\newtheorem{rem}[thm]{Remark}

\numberwithin{equation}{section}

\begin{document}
	
\setcounter{tocdepth}{3}
\contentsmargin{2.55em} 
\dottedcontents{section}[3.8em]{}{2.3em}{.4pc} 
\dottedcontents{subsection}[6.1em]{}{3.2em}{.4pc}
\dottedcontents{subsubsection}[8.4em]{}{4.1em}{.4pc}

\makeatletter
\newcommand{\mycontentsbox}{%

	{\centerline{NOT FOR PUBLICATION}
		\addtolength{\parskip}{-2.3pt}
		\tableofcontents}}
\def\enddoc@text{\ifx\@empty\@translators \else\@settranslators\fi
	\ifx\@empty\addresses \else\@setaddresses\fi
	\newpage\mycontentsbox\newpage\printindex}
\makeatother

\setcounter{page}{1}

\title[Stable nc polynomials and determinantal representations]{Stable noncommutative polynomials and their determinantal representations}

\author[J. Vol\v{c}i\v{c}]{Jurij Vol\v{c}i\v{c}${}^1$}
\address{Jurij Vol\v{c}i\v{c}, Department of Mathematics\\
	Texas A\&M University \\ Texas}
\email{volcic@math.tamu.edu}
\thanks{${}^1$Research supported by the Deutsche Forschungsgemeinschaft (DFG) Grant No. SCHW 1723/1-1.} 

\date{\today}

\begin{abstract}
A noncommutative polynomial is {\it stable} if it is nonsingular on all tuples of matrices whose imaginary parts are positive definite. In this paper a characterization of stable polynomials is given in terms of {\it purely stable} linear matrix pencils, i.e., pencils of the form $H+iP_0+P_1x_1+\cdots+P_dx_d$, where $H$ is hermitian and $P_j$ are positive semidefinite matrices. Namely, a noncommutative polynomial is stable if and only if it admits a determinantal representation with a purely stable pencil. More generally, structure certificates for noncommutative stability are given for linear matrix pencils and noncommutative rational functions.
\end{abstract}

\maketitle


\section{Introduction}

A multivariate polynomial $f\in\C[x_1,\dots,x_d]$ is {\it stable} if 
$f(\alpha)\neq0$ whenever $\imag \alpha_j>0$ for all $j=1,\dots,d$. Stable polynomials and their variations, such as Hurwitz and Schur polynomials, originated in control theory \cite{FB,Bos,Kum,KTM}. However, recent years saw a renewed interest in stable polynomials in a quite wide range of areas \cite{Wag}. A decade ago various problems in combinatorics, matrix theory and statistical mechanics were resolved using stable polynomials, such as the Johnson conjectures \cite{BB0}, new proofs of the Van der Waerden and the Schrijver-Valiant conjectures \cite{Gur}, and Lee-Yang type theorems \cite{BB1}. In real algebraic geometry \cite{BCR,BPT}, stable polynomials emerged through their connection to hyperbolic polynomials \cite{KPV0,JT}, most prominently in the solutions of the Lax conjecture \cite{HV} and the Kadison-Singer conjecture \cite{MSS}. From a complex analysis perspective, stable polynomials are closely related to the Schur-Agler class of rational inner functions \cite{Agl,Kne,GKVVW}.

The common thread of these developments are determinantal representations of stable polynomials using linear matrix pencils with a distinguished structure \cite{Bra,NT}. Namely, if $S$ is a symmetric matrix and $P_1,\dots,P_d$ are positive semidefinite matrices, then
\begin{equation}\label{e:det}
f=\det(S+P_1x_1+\cdots+P_dx_d)\in\R[x_1,\dots,x_d]
\end{equation}
is either zero or a stable polynomial; see e.g. \cite[Proposition 2.4]{BB0}. Conversely, as a consequence of the celebrated Helton-Vinnikov theorem \cite{HV}, every real stable polynomial $f$ in two variables is of the form \eqref{e:det} by \cite[Theorem 5.4]{BB0}. However, the converse fails for polynomials in more than two variables \cite{Bra}. The existence of a special determinantal representation \eqref{e:det} is closely related to having a structural certificate for linear matrix pencils to be invertible on the positive orthant in $\C^d$. Such problems have natural analogs in free analysis and free real algebraic geometry. Here, pencils are evaluated on matrices rather than on scalars, and these new noncommutative problems are often more tractable since matrix evaluations capture the structural properties more completely than just scalar evaluations. For example, Helton \cite[Theorem 1.1]{Hel} showed that every positive noncommutative polynomial is a sum of hermitian squares; also see \cite{HMV,HKM,BPT,KPV} for further results of this flavor. The aim of this paper is to introduce stable noncommutative polynomials and to prove that they admit ``perfect'' determinantal representations. This is achieved by proving a structural theorem for stable linear matrix pencils.

\subsection*{Main results}

Let $\ulx=(x_1,\dots,x_d)$ be freely noncommuting variables. In our noncommutative setting, the positive orthant in $\C^d$ is replaced by the set of all tuples of matrices whose imaginary part is positive definite, which we call the {\it matricial positive orthant} and denote $\allh$. Then we say that a linear matrix pencil $L$ is {\it stable} if $L(X)$ is invertible for every $X\in\allh$. For example, if $H$ is a hermitian matrix and $P_0,\dots,P_d$ are positive semidefinite matrices such that $\ker H\cap\bigcap_j\ker P_j=\{0\}$, then
\begin{equation}\label{e:rs}
L=H+iP_0+P_1x_1+\cdots P_dx_d
\end{equation}
is a stable pencil (Proposition \ref{p:good}). Due to their special structure we call pencils of the form \eqref{e:rs} {\it \sstable}. Our first main result states that every stable pencil is built of \sstable pencils.

\begin{thmA}\label{t:a}
A $\de\times \de$ linear pencil $L$ is stable if and only if there exist $D,E\in\GL_{\de}(\C)$ such that
$$DLE=
\begin{pmatrix}
L_1 & & \\
\star & \ddots & \\
\star & \star & L_\ell
\end{pmatrix},
$$
where $L_1,\dots,L_\ell$ are \sstable pencils.
\end{thmA}

Theorem \ref{t:a} is a special case of Theorem \ref{t:stable} which deals more generally with rectangular pencils. Its proof also yields an algorithm relying on semidefinite programming for checking whether a pencil is stable (Subsection \ref{ss:algo}). Note that Theorem \ref{t:a} is especially intriguing since it represents an algebraic certificate for invertibility on an open matricial set; usually such certificates are obtained for closed (convex) sets \cite{HKM,BMV} or are less clean \cite{KPV}. We also obtain a strengthened version of Theorem \ref{t:a} for hermitian pencils (Proposition \ref{p:herm}), and a size bound for invertibility of linear matrix pencils on the matricial polydisk (Corollary \ref{c:bd}).

Next we characterize noncommutative rational functions that are regular on $\allh$ (Theorem \ref{t:rat}) by combining Theorem \ref{t:a} and realization theory for noncommutative rational functions \cite{BGM,BR}, leading to determinantal representations of stable noncommutative polynomials. We say that $f\in\px$ is {\it stable} if $\det f(X)\neq0$ for all $X\in\allh$.

\begin{thmA}\label{t:b}
Let $f\in\px$. Then $f$ is stable if and only if there exists a \sstable pencil $L$ such that $\det f(X)=\det L(X)$ for all matrix tuples $X$. 
\end{thmA}

See Theorem \ref{t:irr} for the proof. Finally, we consider {\it hermitian} polynomials, which are noncommutative analogs of real polynomials. In contrast to the commutative setting, stable hermitian polynomials display surprisingly rigid behavior.

\begin{thmA}\label{t:c}
Every stable irreducible hermitian polynomial is affine.
\end{thmA}

See Theorem \ref{t:aff} for a more precise statement. The paper concludes with Section \ref{sec4} on possible applications of the derived theory in multidimensional systems and circuits.

\subsection*{Acknowledgments}

The author thanks Andreas Thom for bringing this topic to his attention, and Igor Klep for fruitful suggestions.

\section{Stable pencils}

In this section we completely characterize stable linear matrix pencils, i.e., rectangular pencils that have full rank on the matricial positive orthant. We prove that every such pencil is equivalent to a lower block triangular pencil whose diagonal blocks are stable for obvious reasons (and thus called \sstable pencils). This result is then strengthened for hermitian pencils. Lastly, the characterization is extended to other classical notions of stability.

\subsection{Notation}

We start by introducing the basic terminology used throughout the paper, including \sstable pencils.

\subsubsection{Linear matrix pencils}

For $d\in\N$ let $\ulx=(x_1,\dots,x_d)$ be a tuple of freely noncommuting variables and let $\px$ be the free $\C$-algebra generated by $\ulx$. If $A_0,\dots,A_d\in\C^{\de\times \ve}$, then
$$L=A_0+A_1x_1+\cdots+A_dx_d\in\C^{\de\times \ve}\otimes_{\C}\px=\px\!{}^{\de\times \ve}$$
is a {\bf linear matrix pencil} of size $\de\times \ve$. If $\de=\ve$, then we simply say that $L$ is of size $\de$. If $X\in\mat{n}^d$, then the evaluation of $L$ at $X$ is defined as
$$L(X)=A_0\otimes I+\sum_jA_j\otimes X_j\in\C^{\de n\times \ve n},$$
where $\otimes$ is the Kronecker product. Let us also denote $\all=\bigcup_{n\in\N}\mat{n}^d$.

\subsubsection{Real and imaginary part of a matrix}

Let $\herm{n}\subset\mat{n}$ denote the $\R$-subspace of hermitian matrices. For $X\in\mat{n}$ let
$$\real X=\frac12 (X+X^*),\qquad \imag X=\frac{1}{2i}(X-X^*).$$
Then $\real X,\imag X\in\herm{n}$ and $X=\real X+i\imag X$.

\begin{lem}\label{l:elem}
Let $X\in\mat{n}$ and $\imag X\succeq0$. Then $\ker X=\ker (\real X)\cap \ker (\imag X)$.
\end{lem}

\begin{proof}
The inclusion $\supseteq$ clearly holds. Conversely, let $v\in\ker X$. Then
$$\langle (\real X)v,v\rangle+i \langle (\imag X)v,v\rangle=\langle Xv,v\rangle=0.$$
Now $\langle (\real X)v,v\rangle,\langle (\imag X)v,v\rangle\in\R$ implies
$$\langle (\real X)v,v\rangle=\langle (\imag X)v,v\rangle=0.$$
Since $\imag X\succeq0$, we have $(\imag X)v=0$, and therefore $(\real X)v=X v-i(\imag X)v=0$. Hence $v\in\ker (\real X)\cap \ker (\imag X)$.
\end{proof}

\subsection{Stable pencils}

This subsection introduces stable pencils, which are the core objects of this paper. Then we single out two particular kinds of such pencils that are stable ``for obvious reasons'', \sstable and \strstable pencils.

Let
$$\allh=\bigcup_{n\in\N}\left\{(X_1,\dots,X_d)\in\mat{n}\colon \imag X_j\succ0 \  \forall j\right\}\subset \all$$
be the {\bf matricial positive orthant}. The sets $\allh\cap\mat{n}^d$ are closely related to Siegel upper half-spaces \cite{vdG,JT1}.

\begin{defn}\label{d:stable}
A linear matrix pencil $L$ is {\bf stable} if $L(X)$ has full rank for all $X\in\allh$.
\end{defn}

The next property is the first step towards a structural characterization of stable pencils.

\begin{defn}\label{d:strong}
A pencil $L=H+i P_0+\sum_{j=1}^d P_jx_j$ of size $\de$ is {\bf \sstable} if
$$H\in\herm{\de},\qquad P_j\succeq0 \quad \forall j=0,\dots,d, \qquad \ker H\cap\bigcap_{j=0}^d \ker P_j=\{0\}.$$
\end{defn}

The above terminology is justified by the following proposition.

\begin{prop}\label{p:good}
Every \sstable pencil is stable.
\end{prop}

\begin{proof}
Let $X\in\allh\cap\mat{n}^d$ and let $L=H+i P_0+\sum_{j>0} P_jx_j$ be \sstable. Then
\begin{align*}
L(X) &= H\otimes I+i P_0\otimes I+\sum_{j>0} P_j\otimes X_j \\
&= 
\left(H\otimes I+\sum_{j>0} P_j\otimes \real X_j \right)+
i \left(P_0\otimes I+\sum_{j>0} P_j\otimes \imag X_j\right).
\end{align*}
Note that $\imag L(X)\succeq0$. If $v\in\ker L(X)$, then by Lemma \ref{l:elem} and positive semidefiniteness we have
\begin{align*}
v &\in \ker \left(H\otimes I+\sum_{j>0} P_j\otimes \real X_j \right)
\cap \ker \left(P_0\otimes I+\sum_{j>0} P_j\otimes \imag X_j\right) \\
&= \ker (P_0\otimes I)\cap \left(\bigcap_{j>0}\ker (P_j\otimes \imag X_j)\right)
\cap\ker \left(H\otimes I+\sum_{j>0} P_j\otimes \real X_j \right).
\end{align*}
It is easy to see that $\ker(A\otimes B)=\ker A\otimes \C^n$ for every $A\in\mat{\de}$ and $B\in\gl{n}$. Since $\imag X_j\succ0$, we have $v\in\ker P_j\otimes \C^n$ for all $j$, and consequently $v\in \ker H\otimes \C^n$. Finally, $\ker H\cap\bigcap_j \ker P_j=\{0\}$ implies $v=0$.
\end{proof}

Using \sstable pencils as building blocks, one can produce more stable pencils.

\begin{defn}\label{d:weak}
	Let $L=A_0+\sum_{j>0}A_jx_j$ with $A_j\in\C^{\de\times \ve}$ and $\de\ge \ve$. We temporarily say that $L$ is {\bf \strstable} if
	$$DLE=\begin{pmatrix}
	L_1 & & \\
	\star & \ddots & \\
	\star & \star & L_\ell
	\end{pmatrix}$$
	for some $D\in \C^{\ve\times \de}$, $E\in\C^{\ve\times \ve}$ and \sstable pencils $L_1,\dots,L_\ell$.
	
	If $\de < \ve$, then we call $L$ \strstable if $L^{\rm t}=A_0^{\rm t}+\sum_{j>0}A_j^{\rm t}x_j$ is \strstable.
\end{defn}

\begin{rem}\label{r:def}
The matrices $D$ and $E$ from Definition \ref{d:weak} necessarily have full rank and every \strstable pencil is stable by Proposition \ref{p:good}.
	
Furthermore, if $\ell=1$ in Definition \ref{d:weak}, then $E$ is redundant: if $DLE=L_1$ is \sstable, then $((E^{-1})^*D)L=(E^{-1})^*L_1E^{-1}$ is also \sstable.
\end{rem}

\begin{exa}
Let
$$L=\begin{pmatrix}1+2x_1 & -x_1 \\ -x_1 & -1\end{pmatrix}.$$
Then $L$ is \strstable since
$$
\begin{pmatrix}1 & 1 \\ 1 & 0\end{pmatrix}L\begin{pmatrix}1 & 1 \\ 0 & 1\end{pmatrix}
=\begin{pmatrix}1 & 0 \\ 1 & 1\end{pmatrix}+\begin{pmatrix}1 & 0 \\ 2 & 1\end{pmatrix}x_1.
$$
Suppose that $DL$ is \sstable for some $D\in\gl{2}$. From the $\R$-linear system $\imag(DA_1)=0$ in $D$ we deduce that
$$D=\begin{pmatrix}\alpha_1+i\beta & \alpha_2+2i\beta \\ \alpha_3 & 2\alpha_3+\alpha_1-i\beta\end{pmatrix},\qquad \alpha_j,\beta\in\R.$$
Furthermore $\det\imag(DA_0)=-\frac14(\alpha_2+\alpha_3)^2$, so $\imag(DA_0)\succeq0$ implies $\alpha_3=-\alpha_2$. Then an easy calculation shows that $\det \real(DA_1)=-\det D$, so $\real(DA_1)\succeq0$ contradicts $\det D\neq0$.

Therefore one cannot assume $\ell=1$ in Definition \ref{d:weak} in general.	\bqed
\end{exa}

Let $L=A_0+\sum_{j>0}A_jx_j$ be of size $d$ and $A_0\in\gl{\de}$. Then we say that $L$ is an {\bf \irr} pencil if $A_1A_0^{-1} ,\dots,A_dA_0^{-1}$ generate $\mat{\de}$ as a $\C$-algebra (cf. \cite[Section 3.4]{KV} or \cite[Section 2.1]{HKV})\footnote{
	Where such pencils were called \textit{irreducible}.
	}.

\begin{lem}\label{l:irr}
Let $L$ be an \irr pencil of size $\de$. If $L$ is \strstable, then it has only one \sstable block; that is, $DL$ is \sstable for some $D\in\gl{\de}$.
\end{lem}

\begin{proof}
Let $D,E\in\gl{\de}$ be such that
\begin{equation}\label{e:21}
DLE=\begin{pmatrix}
L_1 & & \\
\star & \ddots & \\
\star & \star & L_\ell
\end{pmatrix},
\end{equation}
where $L_1,\dots,L_\ell$ are \sstable pencils. Then the coefficients of
$$D\big(L L(0)^{-1} \big) D^{-1}=DLE\big(DL(0)E)^{-1}$$
generate $\mat{\de}$; however, they have a block lower triangular form as in \eqref{e:21}, so $\ell=1$. By Remark \ref{r:def} we thus have $DL=L_1$ for some $D\in\gl{\de}$ and a \sstable pencil $L_1$.
\end{proof}

\subsection{Main theorem}

In this subsection we apply a truncated Gelfand-Naimark-Segal (GNS) construction to prove that every stable pencil is \strstable; see Theorem \ref{t:stable}. We start with some preliminary notation.

By $\ulx^*=(x_1^*,\dots,x_d^*)$ we denote the formal adjoints of variables $x_j$ and endow the free algebra $\pxx$ with the corresponding involution. Let $L=A_0+\sum_jA_jx_j$ be a linear pencil of size $\de\times \ve$ and $\de\ge \ve$. For $\ell=0,1,2$ let $V_\ell$ denote the subspace of elements of degree at most $\ell$ in $\mat{\ve}\otimes\pxx$. Furthermore define
\begin{align*}
C_1 &= \left\{\sum_{j>0} P_j \imag x_j\colon P_j\succeq0\ \forall j \right\}\subset V_1, \\
C_2 &= \left\{\sum_k L_k L_k^*\colon L_k\in V_1 \right\}\subset V_2, \\
U &= \C^{\ve\times \de}L+L^*\C^{\de\times \ve}\subset V_1.\\
\end{align*}

The following lemma relies on a variant of an argument that was used to prove the one-sided real Nullstellensatz \cite{CHMN}.

\begin{lem}\label{l:cones}
Keep the notation from above.
\begin{enumerate}
\item $C_1+C_2$ is a closed convex cone in $V_2$.
\item Assume
\begin{equation}\label{e:cond}
\imag(DA_0)\succeq 0 \text{ and } \imag(DA_j)=0,\ \real(DA_j)\succeq 0 \text{ for } j>0 \quad \implies \quad DL=0
\end{equation}
holds for all $D\in \C^{\ve\times \de}$. Then
\begin{enumerate}
\item[(2a)] $U\cap (C_1+C_2)=\{0\}$,
\item[(2b)] there exists $X\in\allh\cap \mat{\ve}$ such that $\ker L(X)\neq\{0\}$.
\end{enumerate}
\end{enumerate}
\end{lem}

\begin{proof}
(1) It is clear that $C_1$ and $C_2$ are convex cones in $V_2$, $C_1$ is closed and $C_1\cap C_2=\{0\}$. Furthermore, using Caratheodory’s theorem on convex hulls \cite[Theorem 17.1]{Roc} it is easy to show that $C_2$ is closed in $V_2$; see e.g. \cite[Proposition 3.1]{HKM}. Therefore $C_1+C_2$ is closed by \cite[Corollary 9.1.3]{Roc}.

(2a) Let $DL+L^*E\in U\cap (C_1+C_2)$. Then $DL+L^*E$ is hermitian and so $2(DL+L^*E)=(D+E)L+L^*(D+E)^*$. Furthermore, $DL+L^*E$ is of degree at most 1 and hence
$$(D+E)L+L^*(D+E)^*=P_0+\sum_j P_j \imag x_j$$
for some $P_j\succeq0$. Then $\real((D+E)A_0)\succeq0$ and $\real((D+E)A_j)=0$, $-\imag((D+E)A_j)\succeq 0$ for $j>0$. For $\tilde{D}=i(D+E)$ we thus have $\imag(\tilde{D}A_0)\succeq0$ and $\imag(\tilde{D}A_j)=0$, $\real(\tilde{D}A_j)\succeq0$ for $j>0$. By the assumption \eqref{e:cond} we have $\tilde{D}L=0$, so $(D+E)L=0$ and therefore $DL+L^*E=0$.

(2b) Let $V'_2\subset V_2$ be the $\R$-subspace of hermitian elements. By (2a) and \cite[Theorem 2.5]{Kle} there exists an $\R$-linear functional $\lambda_0:V_2'\to \R$ satisfying
$$\lambda_0\left((C_1+C_2)\setminus\{0\}\right)=\R_{>0},\qquad \lambda_0(U\cap V'_2)=\{0\}.$$
We extend $\lambda_0$ to $\lambda:V_2\to\C$ as $\lambda(f)=\lambda_0(\real f)+i\lambda_0(\imag f)$. Then $\lambda$ determines a scalar product $\langle f_1,f_2\rangle =\lambda(f_2^*f_1)$ on $V_1$ because $\lambda(C_2\setminus\{0\})=\R_{>0}$. Also note that $\lambda(U)=\{0\}$.

Let $\pi:V_1\to V_0$ be the orthogonal projection. Note that $V_0=\mat{\ve}$. For every $a,v\in \mat{\ve}$ and $1\le j\le d$ we have
$$\langle \pi(ax_j),v\rangle = \langle ax_j,v\rangle=\langle x_j,a^*v\rangle=\langle \pi(x_j),a^*v\rangle =\langle a\pi(x_j),v\rangle$$
and thus
\begin{equation}\label{e:inter}
\pi(a f)=a\pi(f) \qquad \forall a\in \mat{\ve},\ f\in\cV_1.
\end{equation}
For $j=1,\dots,d$ and $a\in\mat{\ve}$ we define operators
\begin{alignat*}{3}
Y_j&\colon V_0\to V_0, &\qquad & f \mapsto \pi(x_j f), \\
\ell_a&\colon V_0\to V_0, &\qquad & f \mapsto a f.
\end{alignat*}
It is easy to see that $\lambda(C_1\setminus\{0\})=\R_{>0}$ implies $Y=(Y_1,\dots, Y_g)\in\allh$. By \eqref{e:inter}, operators $\ell_a$ and $Y_j$ commute. A straightforward argument shows that $\ell_a^*=\ell_{a^*}$, so $\ell_a$ also commute with $Y_j^*$. Furthermore, the map
$$\mat{\ve}\to\End(\cV_0)\cong \mat{\ve}=\mat{\ve}\otimes\mat{\ve}$$
given by $a\mapsto \ell_a$ is a unital $*$-embedding of $*$-algebras. By a $*$-version of the Skolem-Noether theorem \cite[Theorem 11.9]{Tak} there exists a unitary $Q:\cV_0\to \C^{\ve}\otimes \C^{\ve}$ such that
$$Q\ell_a Q^*=a\otimes I$$
for all $a\in\mat{\ve}$. Since $Y_j$ and $Y^*_j$ commute with operators $\ell_a$, there exist $X_j\in\mat{\ve}$ such that
$$QY_j Q^*=I\otimes X_j.$$
Since $Q$ is unitary, we have $QY_j^*Q^*=I\otimes X_j^*$ and consequently $X=(X_1,\dots,X_d)\in\allh$.

Let $D\in\C^{\ve\times \delta}$ be arbitrary and consider the pencil $DL$ of size $\ve$. By the previous paragraph, $(DL)(X)$ can be viewed as an operator on $V_0$ and
$$(DL)(X)=\ell_{DA_0}+\sum_{j>0}\ell_{DA_j}\circ Y_j.$$
If $u\in V_0$ denotes the $\ve\times \ve$ identity matrix, then
$$\langle (DL)(X)u,f\rangle=\langle \pi(DL),f\rangle=\langle DL,f\rangle=\lambda((f^*D)L)=0$$
for all $f\in V_0$ by $\lambda(U)=\{0\}$. Hence $(DL)(X)u=0$ for every $D\in\C^{\ve\times \de}$. Then it is easy to see that $L(X)u=0$ and hence $\ker L(X)\neq\{0\}$.
\end{proof}

\begin{thm}\label{t:stable}
Let $L$ be a linear pencil of size $\de\times \ve$. The following are equivalent:
\begin{enumerate}
\item $L$ is stable;
\item $L$ is \strstable;
\item $L(X)$ has full rank for all $X\in\allh\cap\mat{\min\{\de,\ve\}}^d$.
\end{enumerate}
\end{thm}

\begin{proof}
$(2)\Rightarrow(1)$ is already stated in Remark \ref{r:def}, and $(1)\Rightarrow(3)$ is trivial. Hence we prove $(3)\Rightarrow(2)$.

Without loss of generality let $\de\ge \ve$. We prove the statement by induction on $\ve$ by looking at the solutions $D\in\C^{\ve\times \de}$ of the system
\begin{equation}\label{e:sys}
\imag(DA_0)\succeq 0 \text{ and } \imag(DA_j)=0,\ \real(DA_j)\succeq 0 \text{ for } j>0.
\end{equation}

First let $\ve=1$. If $L$ is not \strstable, then for every $D\in \C^{1\times \de}$, $DL$ is not \sstable. Hence every solution $D$ of \eqref{e:sys} satisfies $DL=0$, so $L(X)$ does not have full rank for some $X\in\allh\cap \C^d$ by Lemma \ref{l:cones}.

Now assume the statement holds for all $\ve'<\ve$ and that $L$ is not \strstable. By composing the coefficients of $L$ on the left with the projection onto $\sum_j \ran A_j$, we can without loss of generality assume that $\sum_j \ran A_j= \C^d$. Since $L$ is not \strstable, $DL$ is in particular not \sstable for any $D\in \C^{\ve\times \de}$, so every solution $D$ of \eqref{e:sys} satisfies
$$\cK:=\bigcap_j \big(\ker (DA_j)\cap \ker (DA_j)^*\big)\neq\{0\}$$
by Lemma \ref{l:elem}.

If every solution $D$ of \eqref{e:sys} satisfies $DL=0$, then $L(X)$ does not have full rank for some $X\in\allh\cap \mat{\ve}^d$ by Lemma \ref{l:cones}.

Otherwise there exists a solution $D$ of \eqref{e:sys} such that $DL\neq0$. Let $Q$ be a $\ve\times\ve$ unitary matrix such that its columns form an orthonormal basis corresponding to the orthogonal decomposition $\C^\ve=\cK^\perp\oplus\cK$, and write
$$Q^*D=\begin{pmatrix}D_1 \\ D_0\end{pmatrix},\qquad LQ=\begin{pmatrix}L_1 & L_0\end{pmatrix}.$$
By the definition of $\cK$ we have
$$Q^*DLQ=\begin{pmatrix}D_1L_1 & 0 \\ 0 & 0\end{pmatrix},$$
where $L_1$ is a \sstable pencil. Therefore $D_1L_0=0$, and $D_0=0$ since $\sum_j \ran A_j= \C^d$. If $L_0$ were \strstable, then there would exist $D_2,E_2$ of appropriate sizes such that $D_2L_0E_2$ would be a block lower triangular matrix with \sstable pencils on the diagonal. Then
$$\begin{pmatrix}D_1 \\ D_2\end{pmatrix}L Q(I\oplus E_2) =
\begin{pmatrix}D_1L_1 & 0 \\ D_2 L_1 & D_2L_0E_2\end{pmatrix}$$
contradicts the assumption that $L$ is not \strstable. Therefore $L_0$ is not \strstable, so by the induction hypothesis there exists $X\in\allh\cap \mat{\ve}^d$ such that $L_0(X)$ does not have full rank. Hence $L(X)$ does not have full rank.
\end{proof}

\subsubsection{An algorithm}\label{ss:algo}

The proof of Theorem \ref{t:stable} can be used to devise an algorithm for testing whether a pencil is stable by solving a sequence of semidefinite programs (SDPs) \cite{BPT,WSV}. 

Let $L=A_0+\sum_{j>0}A_jx_j$ be of size $\de\times \ve$ with $\de\ge \ve$.
\begin{enumerate}
\item Solve the following feasibility SDP for $D\in\C^{\ve\times \de}$:
\begin{equation}\label{e:sdp}
\begin{split}
\imag(DA_0) &\succeq 0\\
\real(DA_j) &\succeq 0\quad \text{ for } j>0\\
\tr\left(\imag(DA_0)+\sum_{j>0}\real(DA_j)\right) &= 1\\
\imag(DA_j) &= 0 \quad \text{ for } j>0.
\end{split}
\end{equation}
\item If \eqref{e:sdp} is infeasible, then $L$ is not stable.
\item Otherwise let $D$ be the output of \eqref{e:sdp} and let $\cK=\bigcap_{j\ge0}\ker(DA_j)$. If $\cK=\{0\}$, then $L$ is stable. If $\cK\neq\{0\}$, then let $V$ be a matrix whose columns form a basis of $\cK$. By the proof of Theorem \ref{t:stable}, $L$ is stable if and only if $LV$ is stable. Then we apply (1) to $LV$ and continue.

This procedure will eventually stop because $LV$ is of smaller size than $L$.
\end{enumerate}

Similar algorithms exist for testing whether a pencil is of full rank on all hermitian tuples \cite{KPV} or on free spectrahedra given by monic hermitian pencils \cite{HKMV}; the latter situation is especially interesting for the study of linear matrix inequalities \cite{BEFB}. However, in both preceding cases there is no clean structural analog of Theorem \ref{t:stable}.

\subsection{Hermitian coefficients}

Classically, one is interested in symmetric or hermitian determinantal representations \eqref{e:det} of real polynomials. However, the constant term of a \sstable pencil is in general not hermitian. This can be amended for a particular class of pencils. We say that $L=H_0+\sum_{j>0}H_jx_j$ is a {\bf hermitian} pencil if $H_j\in\herm{\de}$ for $j=0,\dots,d$.

\begin{prop}\label{p:herm}
Let $L=H_0+\sum_{j>0}H_jx_j$ be hermitian and $D\in\gl{\delta}$.
\begin{enumerate}
\item Assume $\bigcap_{j>0}\ker H_j=\{0\}$. If $DL$ is \sstable, then $DL$ is hermitian.
\item If $L$ is \irr and stable, then $L$ or $-L$ is \sstable.
\end{enumerate}

\end{prop}

\begin{proof}	
(1) Note that $DH_j$ are hermitian and positive semidefinite for $j>0$ because $DL$ is \sstable. First we claim that the eigenvalues of $D^*$ are real. Let $D^*v=\lambda v$ for $v\neq0$. Then for every $j=1,\dots,d$ we have
$$(\lambda-\bar{\lambda}) \langle DH_jv,v\rangle
 = \lambda\langle DH_jv,v\rangle-\bar{\lambda}\langle H_jD^*v,v\rangle
 = \langle \lambda H_jv,D^*v\rangle-\langle H_jD^*v,\lambda v\rangle 
 = 0.$$
Suppose $\langle DH_jv,v\rangle=0$ for all $j>0$. Since $DH_j\succeq0$, we have $DH_jv=0$, so $v=0$ by the assumption, contradicting $v\neq0$. Therefore $\langle DH_jv,v\rangle\neq0$ for some $j>0$ and hence $\lambda=\bar{\lambda}$.

Next we show that $D^*$ is diagonalizable. Let $(D^*-\lambda I)^2v=0$. Then
$$\langle (DH_j)(D^*-\lambda I)v,(D^*-\lambda I)v\rangle
= \langle (D-\lambda I)DH_j(D^*-\lambda I)v,v\rangle
= \langle DH_j(D^*-\lambda I)^2v,v\rangle
=0$$
for all $j>0$. Since $DH_j\succeq0$, it follows that $(DH_j)(D^*-\lambda I)v=0$ for all $j>0$, so $(D^*-\lambda I)v=0$. Therefore $D^*$ is diagonalizable.

If $D^*v=\lambda v$, then
\begin{equation}\label{e:eig2}
2i\langle \imag(DH_0) v,v\rangle
=\langle H_0 v,D^*v\rangle-\langle H_0D^* v,v\rangle=0
\end{equation}
because $\lambda\in\R$. Now if $v_1$ and $v_2$ are eigenvectors for $D^*$, then
$$\langle \imag(DH_0) (v_1\pm v_2),(v_1\pm v_2)\rangle\ge0$$
since $\imag(DH_0)\succeq0$, which together with \eqref{e:eig2} implies
\begin{equation}\label{e:eig3}
\langle \imag(DH_0) v_1,v_2\rangle+\langle \imag(DH_0) v_2,v_1\rangle=0.
\end{equation}
Now let $v\in\C^d$ be arbitrary. Because $D$ is diagonalizable, $v$ can be written as a sum of eigenvectors of $D$. Therefore $\langle \imag(DH_0)v,v\rangle=0$ by \eqref{e:eig3}, and so $\imag(DH_0)=0$.

(2) Since $L$ is \irr and stable, there exists $D$ such that $DL$ is \sstable by Theorem \ref{t:stable} and Lemma \ref{l:irr}. Hence $DL$ is hermitian by (1), i.e., $DH_j=H_jD^*$ for all $j\ge0$. Therefore
$$H_jH_0^{-1}D=H_jD^*H_0^{-1}=DH_jH_0^{-1}$$
for all $j>0$. Since $H_1H_0^{-1},\dots,H_gH_0^{-1}$ generate $\mat{\de}$, it follows that $D$ is a real scalar matrix, so $L$ or $-L$ is \sstable.
\end{proof}

\subsection{Hurwitz and Schur stability}\label{ss:hs}

In control theory, there are also other stability notions, such as Hurwitz and Schur stability, that can be related to Definition \ref{d:stable}. In this subsection we describe how to apply Theorem \ref{t:stable} and the algorithm from Subsection \ref{ss:algo} to test other versions of noncommutative stability.

We say that $L$ is {\bf Hurwitz stable} if $L(X)$ has full rank for every $X$ with $\real(X_j)\succ0$ for $j=1,\dots,d$. Then $L$ is Hurwitz stable if and only if $L(-i x)$ is stable. Therefore one can directly derive the analogs of Theorem \ref{t:stable} and Subsection \ref{ss:algo} for Hurwitz stable pencils.

Let $\|\cdot\|$ denote the spectral norm of matrices, and let
$$\alld=\bigcup_{n\in\N}\left\{(X_1,\dots,X_d)\in\mat{n}^d\colon \|X_j\|<1 \  \forall j\right\}$$
be the {\bf noncommutative polydisk}. We say that $L$ is {\bf Schur stable} if $L(X)$ has full rank for every $X\in\alld$. Using the Cayley transform we see that $L$ is Schur stable if and only if
\begin{equation}\label{e:cay}
L\left( (X_1-iI)(X_1+iI)^{-1},\dots,(X_g-iI)(X_g+iI)^{-1} \right)
\end{equation}
has full rank for all $X\in\allh$. However, \eqref{e:cay} is not a linear matrix pencil anymore. Let $L=A_0+\sum_j A_jx_j$ be of size $\de\times \ve$ with $\de\ge \ve$ and consider the pencil
$$\tilde{L}=\begin{pmatrix}
I(x_1+i) & & & -I \\
 & \ddots & & \vdots \\
 & & I(x_g+i) & -I \\
A_1(x_1-i) & \cdots & A_g(x_g-i) & A_0 \\
\end{pmatrix}$$
of size $(d\ve+\de)\times (d\ve+\ve)$. Using Schur complements it is easy to check that \eqref{e:cay} is invertible if and only if $\tilde{L}(X)$ is invertible. If $A_0$ does not have full rank, then $L$ is not Schur stable. Now let $A_0$ have full rank; i.e., after a left and a right basis change we can assume $A_0=(\begin{smallmatrix}0\\I\end{smallmatrix})$. Let $\tilde{L}_1$ denote the Schur complement of $\tilde{L}$ with respect to the $\ve\times \ve$ block $I$ in $A_0$. Then $\tilde{L}_1$ is a linear matrix pencil of size $(d\ve+\de-\ve)\times (d\ve)$, and $\tilde{L}_1(X)$ is invertible if and only if \eqref{e:cay} is invertible. Therefore $L$ is Schur stable if and only if $\tilde{L}_1$ is stable. In particular, we can test the Schur stability with a sequence of SDPs as in Subsection \ref{ss:algo}. Moreover, Theorem \ref{t:stable} implies the following size bound.

\begin{cor}\label{c:bd}
Let $L$ be a pencil of size $\de\times \ve$. If $L(X)$ has full rank for every $X\in\alld$ of size $d\cdot\min\{\de,\ve\}$, then $L(X)$ has full rank for every $X\in \alld$.
\end{cor}

\begin{rem}
Via realization theory (see Subsection \ref{ss:ncrat} below), Schur stable pencils are closely related to noncommutative rational functions that are regular on the noncommutative polydisk. A particularly interesting subset of such functions is the noncommutative Schur-Agler class. One of its characteristic features is the existence of contractive representations; see \cite{BMV}.
\end{rem}

\section{Stability of noncommutative polynomials}

We are now ready to apply the preceding results to noncommutative polynomials and rational functions. First we characterize noncommutative rational functions whose domains contain the matricial positive orthant $\allh$ (Theorem \ref{t:rat}). Next we show that every stable noncommutative polynomial admits a determinantal representation with a \sstable pencil (Theorem \ref{t:irr}). Finally, we somewhat surprisingly prove that every irreducible hermitian stable polynomial is affine (Theorem \ref{t:aff}).

\subsection{Noncommutative rational functions}\label{ss:ncrat}

After a short introduction of the free skew field and required realization theory, we describe noncommutative rational functions defined on the matricial positive orthant.

\subsubsection{Free skew field}

We give a condensed introduction of noncommutative rational functions using matrix evaluations of formal rational expressions following \cite{KVV2}. Originally they were defined ring-theoretically \cite{Ami,Coh}.

Noncommutative rational expressions are syntactically valid combinations of complex numbers, variables $\ulx$, arithmetic operations $+,\cdot,{}^{-1}$ and parentheses $(,)$. Given a noncommutative rational expression $r$ and $X\in \mat{n}^d$, the evaluation $r(X)\in\mat{n}$ is defined in the obvious way if all inverses appearing in $r$ exist at $X$. The set of all $X\in\all$ such that $r$ is defined at $X$ is is called the domain of $r$ and denoted $\dom r$. On the  set of all expressions with nonempty domains we define an equivalence relation $r_1\sim r_2$ if and only if $r_1(X)=r_2(X)$ for all $X\in\dom r_1\cap \dom r_2$. The equivalence classes with respect to this relation are called {\bf noncommutative rational functions}. By \cite[Proposition 2.1]{KVV2} they form a skew field denoted $\rx$, which is the universal skew field of fractions of $\px$ by \cite[Section 4.5]{Coh}. We define the {\bf domain} of a noncommutative rational function $\rr\in\rx$ as the union of $\dom r$ over all representatives $r$ of $\rr$.

\subsubsection{Realization theory}

Let $\rr\in\rx$ and assume that $\rr$ is regular at the origin, i.e., $0\in\dom\rr$. Then there exist $\de\in\N$, $b,c\in\C^\de$ and a linear pencil $L$ of size $\de$ with $L(0)=I$, such that
\begin{equation}\label{e:11}
	\rr=c^* L^{-1}b.
\end{equation}
We say that \eqref{e:11} is a {\bf (descriptor) realization} of $\rr$ of size $\de$; see \cite[Section 12]{BGM} and \cite{HMV,Vol1}. In automata theory, such realizations are also called linear representations \cite{BR}.

\begin{rem}\label{r:facts}
In general, $\rr$ admits various realizations. Those of the smallest size are called {\bf minimal}, and possess distinguished properties that we now outline. Let $c^* L^{-1}b$ with $L=I-\sum_jA_jx_j$ be a minimal realization of $\rr$ of size $\de$.

\begin{enumerate}[(1)]

\item It is easy to see that
$$c^*v=0 \text{ and } A_jv=0\ \forall j \quad \implies \quad v=0$$
and
$$v^*b=0 \text{ and } v^*A_j=0\ \forall j \quad \implies \quad v=0$$
holds for every $v\in\C^\de$. This observation is also a consequence of a stronger result stating that minimal realizations are controllable and observable \cite[Theorem 9.1]{BGM} (cf. reduced \cite[Proposition 2.1]{BR}).

\item $\rr\in\px$ if and only if $A_1,\dots,A_d$ are jointly nilpotent by \cite[Proposition 2.1]{CR}.

\item By \cite[Theorem 3.1]{KVV1} and \cite[Theorem 3.10]{Vol} we have
$$\dom\rr=\big\{X\in \all\colon \det L(X)\neq0\big\}.$$
	
\item Assume that $\rr$ is {\bf hermitian}, i.e., $\rr(X)^*=\rr(X^*)$ for all $X\in\dom\rr$. Then $\rr$ admits a minimal realization that is hermitian,
$$\rr=c^* \left(H_0+\sum_jH_jx_j\right)^{-1}c=(H_0^{-1}c)^* \left(I+\sum_jH_jH_0^{-1}x_j\right)^{-1}c,$$
where $H_j\in\herm{\de}$; see \cite[Lemma 4.1]{HMV} or \cite[Theorem 6.8]{Vol1}.
	
\end{enumerate}
\end{rem}

\subsubsection{Rational functions on the matricial positive orthant}

We can now apply Theorem \ref{t:stable} to noncommutative rational functions via realization theory.

\begin{thm}\label{t:rat}
Let $\rr\in\rx$. Then $\dom\rr\supset \allh$ if and only if $\rr=c^* L^{-1}b$ for some \strstable pencil $L$.
\end{thm}

\begin{proof}
$(\Leftarrow)$ Let $\rr=c^* L^{-1}b$ for a stable pencil $L$. The matrix $L(\beta)$ is invertible for every $\beta$ in the positive orthant of $\C^d$ by stability of $L$, so the complex polynomial $\det L(z)$ in $d$ commuting variables is nonzero. Hence there exists $\alpha\in\R^d$ such that $\det L(\alpha)\neq0$. Then $0\in\dom\rr(x+\alpha)$ and
$$\rr(x+\alpha)=
c^*\big(L(x+\alpha)\big)^{-1}b=
c^*\big(I+L(\alpha)^{-1}(L-I) \big)^{-1}L(\alpha)^{-1}b$$
is a realization of $\rr(x+\alpha)$. By Remark \ref{r:facts}(3) and stability of $L$ we have
$$
\dom\rr(x+\alpha)=
\big\{X\in \all\colon \det L(X+\alpha I)\neq0\big\}\supseteq\allh,
$$
and consequently $\dom\rr\supset\allh$ since $\alpha\in\R^d$.

$(\Rightarrow)$ Let $\dom\rr\supset\allh$. Since $\dom\rr\cap \C^d$ is a nonempty Zariski open set and $\R^d$ is a Zariski dense set in $\C^d$, there exists $\alpha\in\dom\rr\cap\R^d$. Note that $\rr(x+\alpha)$ again satisfies $\dom\rr(x+\alpha)\supset\allh$. If $c^*L^{-1}b$ is a minimal realization of $\rr(x+\alpha)$, then $L$ is stable by Remark \ref{r:facts}(3). Hence $L(x-\alpha)$ is stable and thus \strstable by Theorem \ref{t:stable} and $\rr=c^*L(x-\alpha)^{-1}b$.
\end{proof}

For later use we record two well-known determinantal identities.

\begin{lem}\label{l:wk}
Let $P\in\GL_{\de}(\C)$ and $u,v\in\C^{\de\times \ve}$. Then
$$\det(P+uv^*)=\det(I+v^*P^{-1}u)\det P,\qquad
\det\begin{pmatrix}0 & u^* \\ u & P\end{pmatrix}=\det(-u^*P^{-1}u)\det P.$$
\end{lem}

Let $\rr\in\rx$ and $0\in\dom\rr$. We say that $\rr$ is {\bf indecomposable} \cite[Section 4.2]{KV} if the pencil appearing in its minimal realization is \irr. We record the following property of hermitian indecomposable functions.

\begin{prop}\label{p:irrrat}
Let $\rr\in\rx$ be hermitian and indecomposable. If $\dom\rr\supset\allh$, then $\dom\rr^{-1}\supset\allh$.
\end{prop}

\begin{proof}
By the assumption and Remark \ref{r:facts}(4), $\rr$ admits a minimal realization $c^*L^{-1}c$ with $L$ hermitian and \irr. Since $\dom\rr\supset\allh$, the proof of $(\Rightarrow)$ in Theorem \ref{t:rat} shows that $L$ is stable, so $L$ or $-L$ is \sstable by Proposition \ref{p:herm}. By Lemma \ref{l:wk} we have
\begin{equation}\label{e:16}
\det\left(-(c^*L^{-1}c)(X)\right)\det L(X) =
\det\left(\begin{pmatrix}0 & c^* \\ c & L\end{pmatrix}(X)\right).
\end{equation}
Note that  $L$ and $(\begin{smallmatrix}0 & c^* \\ c & L\end{smallmatrix})$ or their negatives are \sstable pencils, so
$\det\rr(X)\neq0$ for all $X\in\allh$. Therefore $\dom\rr^{-1}\supset\allh$.
\end{proof}

\begin{rem}
The conclusion of Proposition \ref{p:irrrat}(1) fails in general if $\rr$ is not hermitian and indecomposable; for example, consider $\rr=x_1-i$ and $\rr=1+x_1^2$.
\end{rem}

\subsection{Stable noncommutative polynomials}

We say that $f\in\px$ is {\bf stable} if $f(X)$ is invertible for every $f\in\allh$. That is, $f$ is stable if and only if $\dom f^{-1}\supset\allh$. In this subsection we prove that every stable noncommutative polynomial admits a determinantal representation with a \sstable pencil, see Theorem \ref{t:irr}. Then we turn our attention to hermitian stable polynomials, which are noncommutative analogs of real stable polynomials. Quite contrary to the commutative setting, we show that every irreducible hermitian stable polynomial is affine (Theorem \ref{t:aff}). Here $f\in\px$ is {\bf irreducible} if it cannot be written as $f=f_1f_2$ for some $f_1,f_2\in\px\setminus\C$. 

The following lemma is a descriptor realization analog of \cite[Lemma 5.3]{HKV} (which deals with Fornasini-Marchesini realizations).

\begin{lem}\label{l:FM}
Let $f\in\px$ and $f(0)=1$. If $c^*L^{-1}b$ is a minimal realization of $f^{-1}$ with $L=I-\sum_jA_jx_j$ of size $\de$, then
\begin{enumerate}
\item $f$ admits a minimal realization
\begin{equation}\label{e:inv}
\begin{pmatrix}-c^* & 1\end{pmatrix}
\begin{pmatrix}I-\sum_j A_j(I-bc^*) x_j & -\sum_j (A_jb) x_j \\ 0 & 1\end{pmatrix}^{-1}
\begin{pmatrix}0 \\ 1\end{pmatrix}
\end{equation}
of size $\delta+1$,
\item $\sum_j \ran A_j=\C^\de$ and $\bigcap_j \ker A_j=\{0\}$,
\item $\det f(X)=\det L(X)$ for all $X\in\all$,
\item $L$ is \irr if $f$ is irreducible.
\end{enumerate}
\end{lem}

\begin{proof}
Let $\rr\in\rx$ satisfy $0\in\dom\rr$ and $\rr(0)=1$. If $c^*(I-\sum_jA_jx_j)^{-1}b$ is a realization of $\rr$ of size $\de$, then \eqref{e:inv} is a realization of $\rr^{-1}$ of size $\de+1$, see e.g. \cite[Theorem 3.10]{Vol}. In particular, if $\rr$ admits a minimal realization of size $\de$, then $\rr^{-1}$ admits a minimal realization of size at least $\de-1$. Now assume $\rr\in\px$ and let $c^*(I-\sum_jA_jx_j)^{-1}b$ be its minimal realization. Then $A_1,\dots,A_d$ are jointly nilpotent matrices by Remark \ref{r:facts}(3). Hence there exists $v_1\neq0$ such that $v_1^*A_j=0$ for all $j$. Moreover, joint nilpotency implies $\bigcap_j\ker A_j\neq\{0\}$, and by Remark \ref{r:facts}(1) there exists $v_2\in\bigcap_j\ker A_j$ such that $c^*v_2=1$. Let $V\in\C^{(\de+1)\times (\de-1)}$ be a matrix whose columns form an orthonormal basis of the orthogonal complement of $\{ (\begin{smallmatrix} v_1 \\ 0 \end{smallmatrix}),(\begin{smallmatrix} v_2 \\ 1 \end{smallmatrix}) \}$ in $\C^{\de+1}$. Combining
$$
\begin{pmatrix}v_1^* & 0\end{pmatrix}\begin{pmatrix} 0 \\ 1 \end{pmatrix}=0,\quad
\begin{pmatrix}v_1^* & 0\end{pmatrix}\begin{pmatrix}A_j(I-bc^*) & A_jb \\ 0 & 0\end{pmatrix}=0
$$
and
$$
\begin{pmatrix}-c^* & 1\end{pmatrix}\begin{pmatrix} v_2 \\ 1 \end{pmatrix}=0,\quad
\begin{pmatrix}A_j(I-bc^*) & A_jb \\ 0 & 0\end{pmatrix}\begin{pmatrix} v_2 \\ 1 \end{pmatrix}=0
$$
with \eqref{e:inv} it is easy to see that $\rr^{-1}$ admits a realization
$$
\left(\begin{pmatrix}-c^* & 1\end{pmatrix}V\right)
\left(V^*
\begin{pmatrix}I-\sum_j A_j(I-bc^*) x_j & -\sum_j (A_jb) x_j \\ 0 & 1\end{pmatrix}
V\right)^{-1}
\left(V^*\begin{pmatrix}0 \\ 1\end{pmatrix}\right)
$$
of size $\de-1$.

(1)	Now fix $f\in\px$ with $f(0)=1$, and let $c^*L^{-1}b$ be a minimal realization of $f^{-1}$ with $L=I-\sum_jA_j x_j$. If $f$ admitted a minimal realization of size at most $\de$, then $f^{-1}$ would admit a minimal realization of size at most $\de-1$ by the previous paragraph, which contradicts the assumption on size of $c^*L^{-1}b$. Therefore $f$ admits a minimal realization of size $\de+1$ of the form \eqref{e:inv}.

(2) We have just seen that \eqref{e:inv} is a minimal realization of $f$. If $v\in\C^d$ is such that $v^*A_j=0$ for all $j$, then
$$
\begin{pmatrix}v^* & 0\end{pmatrix}\begin{pmatrix} 0 \\ 1 \end{pmatrix}=0,\quad
\begin{pmatrix}v^* & 0\end{pmatrix}\begin{pmatrix}A_j(I-bc^*) & A_jb \\ 0 & 0\end{pmatrix}=0,
$$
so Remark \ref{r:facts}(1) implies $v=0$, and hence $\sum_j\ran A_j=\C^\de$. Similarly we obtain $\bigcap_j\ker A_j=\{0\}$. 

(3) Matrices $A_j(I-bc^*)$, which by \eqref{e:inv} appear in a minimal realization of $f$, are jointly nilpotent by Remark \ref{r:facts}(2). Next, $c^*b=f(0)^{-1}=1$ implies
$$f^{-1}=c^* L^{-1}b=1+c^*L^{-1}(I-L)b.$$
By Lemma \ref{l:wk} we then have
$$\det L(X)\cdot \det f(X)^{-1}
=\det\big( (L+(I-L)bc^*)(X) \big)
$$
for every $X\in\dom f^{-1}$. But
$$\det\big( (L+(I-L)bc^*)(X) \big)=
\det\left(I\otimes I-\sum_jA_j(I-bc^*)\otimes X_j\right)
=1$$
since $A_j(I-bc^*)$ are jointly nilpotent, so $\det f(X)=\det L(X)$ for all $X\in\all$.

(4) Assume $f$ is irreducible. For $X^{(n)}\in\mat{n}^g$ one can view $\det f(X^{(n)})$ as a polynomial in $gn^2$ variables, By \cite[Theorem 4.3]{HKV}, there exists $n_0\in\N$ such that $\det L(X^{(n_0)})=\det f(X^{(n_0)})$ is an irreducible polynomial for all $n\ge n_0$. Using the results of \cite[Subsection 2.1]{HKV} it is easy to see that $L$ is \irr or $\sum_j\ran A_j\neq\C^\de$ or $\bigcap_j\ker A_j\neq\{0\}$. Therefore $L$ is \irr by (2).
\end{proof}

\begin{thm}\label{t:irr}
Let $f\in\px$. Then $f$ is stable if and only if there exists a \sstable pencil $L$ such that $\det f(X)= \det L(X)$ for all $X\in\all$.
\end{thm}

\begin{proof}
The implication $(\Leftarrow)$ trivially holds, so we consider $(\Rightarrow)$.
Since \sstable pencils are preserved under shifts along $\R^d$ and direct sums, it suffices to assume that $f$ is irreducible and $f(0)\neq0$. Let $f^{-1}=c^* \tilde{L}^{-1}b$ be a minimal realization. The monic pencil $\tilde{L}$ is \irr and
$$\det f(X)=\det \big(f(0)\tilde{L}(X)\big)$$
for all $X\in\all$ by Lemma \ref{l:FM}(3)\&(4). If $f$ is stable, then $\tilde{L}$ is \strstable by Theorem \ref{t:rat}. By Lemma \ref{l:irr} there exists an invertible $D$ such that $f(0)D\tilde{L}$ is \sstable. If $\det D\in\R_{>0}$, then take $L:=(\det D)^{-1}f(0) D\tilde{L}$. Otherwise there is $\gamma\in\R+i\R_{>0}$ such that $\det D=\gamma^2$. Then $L:=(\gamma^{-1} I_2)\oplus (f(0)D\tilde{L})$ is the desired \sstable pencil.
\end{proof}

\begin{exa}\label{exa:lot}
If $\beta\in\R$, then a short calculation shows that
\begin{equation}\label{e:30}
\imag\left((T-\beta I)^{-1}\right)=-(T-\beta I)^{-1}(\imag T)(T^*-\beta I)^{-1}
\end{equation}
for all the matrices $T$ with $T-\beta I$ invertible. Now let $r\in\R(t)$ be an arbitrary univariate rational function of the form
$$r=\sum_{k=1}^\ell \frac{\alpha_k}{t-\beta_k},\qquad \alpha_k\in \R_{<0},\ \beta_k\in\R.$$
If $\imag T\succ 0$, then $(T-\beta_k I)$ is invertible for all $k$, and thus
\begin{equation}\label{e:31}
\imag T\succ 0\quad\Rightarrow \quad  \imag(r(T))\succ 0
\end{equation}
by \eqref{e:30}. Write $r=p/q$ for coprime $p,q\in\R[t]$ and let $f=p(x_1)+q(x_1)x_2\in\px$. Then
$$f(X)=p(X_1)+q(X_1)X_2=q(X_1)\left(q(X_1)^{-1}p(X_1)+X_2\right)$$
is invertible for every $X\in\mathbb{H}^2$ because $\imag(q(X_1)^{-1}p(X_1)+X_2)\succ0$ by \eqref{e:31}. Therefore $f$ is stable and irreducible. \bqed
\end{exa}

\subsubsection{Stable hermitian polynomials}

As for noncommutative rational functions, we say that $f\in\px$ {\bf hermitian} if $f(X)^*=f(X^*)$ for all $X\in\all$. Recall that there exist irreducible stable polynomials of arbitrary degree (Example \ref{exa:lot}). On the other hand, this is not true for hermitian polynomials.

\begin{thm}\label{t:aff}
Let $f\in\px$ be hermitian and irreducible, and $f(0)=1$. Then $f$ is stable if and only if $f=1\pm (\alpha_1x_1+\cdots+\alpha_dx_d)$ for $\alpha_j\in\R_{\ge0}$.
\end{thm}

\begin{proof}
Since $(\Leftarrow)$ is clear, let us prove $(\Rightarrow)$. By Remark \ref{r:facts}(4), $f^{-1}$ admits a hermitian minimal realization $c^*L^{-1}c$
of size $\de$ with $L=H_0+\sum_{j>0} H_jx_j$. Note that
$$c^*L^{-1}c=(H_0^{-1}c)^*(LH_0^{-1})^{-1}c.$$
Then $L$ is \irr by Lemma \ref{l:FM}(4). Furthermore $L$ is stable because $f$ is stable. Moreover, since $L$ is hermitian, $L$ or $-L$ is \sstable by Proposition \ref{p:herm}(2).

By Lemma \ref{l:FM}(1),
\begin{equation}\label{e:nilp}
\begin{pmatrix}-(H_0^{-1}c)^* & 1\end{pmatrix}
\begin{pmatrix}I+\sum_j H_jH_0^{-1}\left(I-c(H_0^{-1}c)^*\right) x_j & \sum_j (H_jH_0^{-1}c) x_j \\ 0 & 1\end{pmatrix}^{-1}
\begin{pmatrix}0 \\ 1\end{pmatrix}
\end{equation}
is a minimal realization for $f$. Denote
$$T=H_0^{-1}-(H_0^{-1}c)(H_0^{-1}c)^*.$$
Therefore $H_jH_0^{-1}\left(I-c(H_0^{-1}c)^*\right)=H_jT$ are jointly nilpotent matrices by Remark \ref{r:facts}(2). Observe that $c^*H_0^{-1}c=f(0)^{-1}=1$ implies $Tc=0$. With respect to an orthonormal basis of $\C^\de$ such that $c$ is a multiple of $e_1$ we have
$$H_j=\begin{pmatrix}\alpha_j & u_j^* \\ u_j & P_j\end{pmatrix},\qquad 
T=\begin{pmatrix}0 & 0 \\ 0 & S\end{pmatrix}$$
for $P_j,S\in\herm{\de-1}$, $u_j\in\C^{\de-1}$ and $\alpha_j\in\R$. Since $L$ or $-L$ is \sstable, we can without loss of generality assume that $H_j\succeq0$ for all $j$. Therefore
$$P_j=V_jV_j^*,\quad u_j=V_jv_j$$
for some $V_j\in\mat{\de-1}$ and $v_j\in\C^{\de-1}$. Furthermore, $P_1S,\dots,P_gS$ are jointly nilpotent $(\de-1)\times(\de-1)$ matrices. Choose $1\le j,k\le d$ and denote $M=V_j^*SV_k$. Then
$$M^*(MM^*)^{\de-1}=V_k^*S (P_jSP_kS)^{\de-1}V_j=0$$
and therefore $M=0$. Consequently
\begin{equation}\label{e:nilp2}
H_jTH_k=0\qquad \forall j,k.
\end{equation}
Combining \eqref{e:nilp} and \eqref{e:nilp2} yields
\begin{align*}
f
&=\begin{pmatrix}-(H_0^{-1}c)^* & 1\end{pmatrix}
\begin{pmatrix}I+\sum_j H_jT x_j & \sum_j (H_jH_0^{-1}c) x_j \\ 0 & 1\end{pmatrix}^{-1}
\begin{pmatrix}0 \\ 1\end{pmatrix} \\
&=\begin{pmatrix}-(H_0^{-1}c)^* & 1\end{pmatrix}
\begin{pmatrix}I-\sum_j H_jT x_j & -\left(I-\sum_j H_jT x_j\right)\left(\sum_j H_jx_j\right)H_0^{-1}c \\ 0 & 1\end{pmatrix}
\begin{pmatrix}0 \\ 1\end{pmatrix} \\
&=1+c^*H_0^{-1}\Bigg(
\left(I-\sum_j H_jT x_j\right)\left(\sum_j H_jx_j\right)
\Bigg)H_0^{-1}c \\
&=1+c^*H_0^{-1}\left(\sum_j H_jx_j\right)H_0^{-1}c \\
&=1+\sum_j\left(c^*H_0^{-1}H_jH_0^{-1}c\right)x_j.
\end{align*}
Finally, since $L$ or $-L$ is \sstable, the real numbers $c^*H_0^{-1}H_jH_0^{-1}c$ have the same sign.
\end{proof}

\section{Applications}\label{sec4}

Theoretical results of previous sections can be applied to multidimensional circuits and systems \cite{Bos1,Bos2}. In engineering, one seeks a controlled system output, and is thus interested in stable systems. Given a $d$-dimensional linear time-invariant system, its stability is related to the zero locus of the denominator $f\in\C[z_1,\dots,z_d]$ of its characteristic or transfer function. In the continuous case, stability corresponds to $f$ not having zeros in the open poly-right-halfplane ($f$ is a Hurwitz polynomial), while in the discrete case stability relates to $f$ not having zeros in the open polydisk ($f$ is a Schur polynomial). For example, consider the discrete Roesser state-space model
\begin{equation}\label{e:roe}
\begin{split}
\begin{pmatrix}
x_1(i_1+1,\dots,i_d) \\ \vdots \\ x_d(i_1,\dots,i_d+1)
\end{pmatrix} &= 
A\begin{pmatrix}
x_1(i_1,\dots,i_d) \\ \vdots \\ x_d(i_1,\dots,i_d)
\end{pmatrix} +B u(i_1,\dots,i_d),\\
y_1(i_1,\dots,i_d) &= 
C\begin{pmatrix}
x_1(i_1,\dots,i_d) \\ \vdots \\ x_d(i_1,\dots,i_d)
\end{pmatrix} +D u(i_1,\dots,i_d)
\end{split}
\end{equation}
as in \cite{Bas}, where $x_j,u,y$ are the state, input and output vectors, respectively, and $A,B,C,D$ are constant matrices of appropriate sizes. Then the denominator of the transfer function for \eqref{e:roe} equals
\begin{equation}\label{e:roe1}
\det\left(
I-A\left(\bigoplus_j I_{\delta_j}z_j\right)
\right),
\end{equation}
where $\delta_j$ is the dimension of $x_j$. Hence one would like to test whether \eqref{e:roe1} is Schur. While there are certain procedures for checking the Schur or the Hurwitz property and their variations \cite{FB,RR}, they are computationally challenging since determining whether a polynomial has a zero in an open domain in $\C^d$ is a hard problem.

Hence we propose the following relaxation. Returning to the model \eqref{e:roe}, one can first ask if the pencil
\begin{equation}\label{e:roe2}
I-A\left(\bigoplus_j I_{\delta_j}x_j\right)
\end{equation}
is Schur stable as in Subsection \ref{ss:hs}. This can be done efficiently using the algorithm from Subsection \ref{ss:algo}. If \eqref{e:roe2} is a Schur stable pencil, then \eqref{e:roe1} is a Schur polynomial. While the converse fails in general, this relaxation is reasonable when it can be hypothesized that the stability of \eqref{e:roe} strongly depends on a specific structure of the matrix $A$. Namely, the algorithm from Subsection \ref{ss:algo} affirms or dismisses this hypothesis.

Such relaxations are also applicable to other models of multidimensional linear systems $\cS$, whose stability can be translated into the Hurwitz/Schur property of the determinant of the pencil arising from matrices in $\cS$. Furthermore, the development of the noncommutative linear systems theory \cite{BGM,BGM1,BGM2} might provide even more direct applications of the results in this paper.


\end{document}